\documentclass[a4paper,reqno,oneside]{amsart}
\usepackage{amsmath,amssymb,amsthm,amsfonts,mathrsfs,amsopn} 
\usepackage{enumitem,dsfont}
\usepackage{xcolor}
\usepackage{tikz-cd}
\usepackage{amssymb}

\usepackage[colorlinks=true,urlcolor=blue, citecolor=red,linkcolor=blue,
linktocpage,pdfpagelabels, bookmarksnumbered,bookmarksopen]{hyperref}

\usepackage{xcolor}

\usepackage{geometry}
\newtheorem{thm}{Theorem}

\newtheorem{prop}[thm]{Proposition}

\theoremstyle{definition}
\newtheorem{defin}[thm]{Definition}

\newtheorem{rem}[thm]{Remark}

\newcommand{\Str}{\mathsf{Str}}
\newcommand{\eps}{\varepsilon}

\newcommand{\C}{\mathbb{C}}
\newcommand{\R}{\mathbb{R}}

\newcommand{\N}{\mathbb{N}}

\newcommand{\cT}{\mathcal{T}}

\newcommand{\cL}{\mathcal{L}}

\newcommand{\cH}{\mathcal{H}}
\newcommand{\cD}{\mathcal{D}}
\newcommand{\cE}{\mathcal{E}}

\DeclareMathOperator{\dom}{dom}

\title[]{Spectral asymptotics for two-dimensional Dirac operators in thin waveguides}

\author{William Borrelli$^*$}
\address[W. Borrelli]{Dipartimento di Matematica, Politecnico di Milano, Via Bonardi, 9, I-20133, Milano, Italy.}
\email{william.borrelli@polimi.it}
\urladdr{}

\author{Nour Kerraoui}
\address[N. Kerraoui]{Aix-Marseille Universit\'e, CNRS, Centrale Marseille, I2M, Marseille, France.}
\email{nour-el-houda.kerraoui@etu.univ-amu.fr}
\urladdr{}

\author{Thomas Ourmi\`eres-Bonafos}
\address[T. Ourmi\`eres-Bonafos]{Aix-Marseille Universit\'e, CNRS, Centrale Marseille, I2M, Marseille, France.}

\email{thomas.ourmieres-bonafos@univ-amu.fr}
\urladdr{http://www.i2m.univ-amu.fr/perso/thomas.ourmieres-bonafos/}

\begin{document}
\thanks{$^*$\emph{Corresponding author}. Dipartimento di Matematica, Politecnico di Milano, Via Bonardi, 9, I-20133, Milano, Italy. E-mail: william.borrelli@polimi.it }
\maketitle
\begin{abstract} We consider the two-dimensional Dirac operator with infinite mass boundary conditions posed in a tubular neighborhood of a $C^4$-planar curve. Under generic assumptions on its curvature $\kappa$, we prove that in the thin-width regime the splitting of the eigenvalues is driven by the one dimensional Schrödinger operator on $L^2(\R)$
\[
	\cL_e := -\frac{d^2}{ds^2} - \frac{\kappa^2}{\pi^2}
\]
with a geometrically induced potential. The eigenvalues are shown to be at distance of order $\varepsilon$ from the essential spectrum, where $2\varepsilon$ is the width of the waveguide. This is in contrast with the non-relativistic counterpart of this model, for which they are known to be at a finite distance.
\end{abstract}

\medskip

{\footnotesize

\medskip

}
\section{Introduction and main result}
\subsection{Introduction}
In this article we continue the study of spectral properties of relativistic quantum waveguides, initiated in \cite{BBKOB22}. In particular, as explained below, we focus on the existence of discrete eigenvalues in the spectral gap, in the thin-width regime.
\smallskip

The study of non-relativistic quantum waveguides started with the pioneering paper~\cite{ES} 
( see also \cite{DE95,GJ,KKriz} for further improvements), where
it was demonstrated that the quantum free Hamiltonian on a waveguide
given by the Dirichlet Laplacian
possesses discrete eigenvalues when the base curve is not a straight line. 
Roughly speaking, 
the corresponding particle gets trapped in any non-trivially curved quantum waveguide. Notice that this is in sharp contrast with the classical case, considering particles following Newton's law with regular reflection at the boundary. Indeed, except for a set of initial conditions of zero measure in the phase space, particles will eventually leave any bounded region in finite time.
Existence and properties of the geometrically induced 
bound states have attracted a lot of attention in the last decades
and this research field is still very active.
We refer the reader to the monograph~\cite{Exner-Kovarik} for a comprehensive discussion of the subject.
\smallskip

The study of the relativistic counterpart of this Hamiltonian started very recently in the two-dimensional case, in \cite{BBKOB22}, considering the Dirac operator on a tubular neighborhood of a curve with \emph{infinite mass} boundary conditions at the boundary. Generally speaking, the mathematical study of such operator on domains started recently \cite{Arrizabalaga-LeTreust-Raymond17,
Barbaroux-Cornean-LeTreust-Stockmeyer_2019, 
Benguria-Fournais-Stockmeyer-Bosch_2017b,LeTreust-Ourmieres-Bonafos_2018},  motivated by models of hadrons confinement from high-energy physics \cite{CJKTW} or by the description of graphene samples \cite{AB}. We also mention the work \cite{Exner-Holzmann}, where spectral properties of Dirac operators on tubes with \emph{ziz-zag} type boundary conditions are considered.

Notice, however, that boundary conditions for Dirac operators on manifolds with boundary had been already considered previously in the geometry literature, see e.g. \cite{G,HMZ} and references therein.
\smallskip

In \cite[Thm.2]{BBKOB22}, under suitable assumptions, it has been proved that the Dirac operator with infinite mass boundary conditions (see \eqref{eq:operator}), posed in the tubular neighborhood of a planar curve, is self-adjoint and its essential spectrum has been identified. Thus, a natural question is to understand the interplay between the geometry and the relativistic setting. In particular, we focus on the existence of geometrically induced bound states in the thin-waveguide regime. For the Dirichlet Laplacian, it is known that in this regime, up to a renormalization factor, the splitting of the eigenvalues is given by an effective operator and this operator is the one-dimensional Schr\"odinger operator with the attractive  potential given by
\[
-\frac{d^2}{ds^2}-\frac{\kappa^2}{4}\,,
\]
where $\kappa$ is the curvature of the underlying curve $\Gamma$, and $s\in\R$ is the arc-length parameter. For this reason one speaks of geometrically induced bound states, related to the non-trivial geometry of the curve/waveguide (see, e.g. \cite{DE95,Exner-Kovarik} ). 
\smallskip

On the other hand, in \cite[Thm.4]{BBKOB22} it is proved that the Dirac operator with infinite mass boundary conditions (see \eqref{eq:operator}), after a suitable choice of renormalization, converges in norm-resolvent sense to a one-dimensional \emph{free} effective Dirac operator 
\[
-i\sigma_1\partial_s+\frac{2}{\pi}m\sigma_3
\]
whose spectrum is purely absolutely continuous. Here $\sigma_1$ and $\sigma_3$ are the first and third Pauli matrices, respectively (as in \eqref{eq:Pauli}). Then in this case, the effective operator does not bear any geometrical information and geometric effects are expected to appear at the next order in the asymptotic expansion in the thin-waveguide regime. The purpose of this paper is precisely to investigate this problem and in our main result Theorem \ref{thm:main} we provide an asymptotic expansion \eqref{eq:evasymp} for the eigenvalues which provides both the splitting and exhibits an effective operator involving the geometry of the underlying curve. This is achieved using min-max techniques, working on the square of the operator \eqref{eq:operator} and relating its eigenvalues to those of a reference operator, defined using its quadratic form \eqref{eq:effectiveform}.

\subsection{Main result}

Let $\gamma : \R \to \R^2$ be an arc-length parametrization of a $C^4$-planar curve $\Gamma$. For $s\in \R$, we define the normal $\nu(s)$ at the point $\gamma(s) \in \Gamma$ such that $(\gamma'(s),\nu(s))$ is an orthonormal basis of $\R^2$. We define the curvature $\kappa(s)$ of the curve $\Gamma$ at the point $\gamma(s)$ by
\[
	\kappa(s) := \gamma''(s) \cdot \nu(s).
\]
Remark that under the smoothness assumption on $\gamma$, $\kappa \in C^2(\R)$ and all along this paper we assume the following hypothesis:
\begin{enumerate}[label=(\Alph*)]
	\item $\lim_{s\to \pm \infty}\kappa(s) = 0$,
	\item $\kappa',\kappa''\in L^\infty(\R)$.
\end{enumerate}

Define the strip $\Str := \R \times (-1,1)$ and $\eps_0 := \|\kappa\|_{L^\infty(\R)}^{-1}$. For $\eps > 0$, we consider the map
\[
	\Phi_\eps : 	\left\{
					\begin{array}{lcl}
						\Str & \to & \R^2\\
						(s,t) & \mapsto & \gamma(s) + \eps t \nu(s).
					\end{array}
				\right.
\]
and define the tubular neighborhood of $\Gamma$
\[
	\Omega_\eps := \Phi_\eps(\Str).
\]
Thus $s\in\R$ and $t\in(-1,1)$ are the arc-length parameter of the curve and the transverse coordinate with respect to the curve, respectively.

In order to guarantee that $\Phi_\eps$ is a $C^3$-diffeomorphism from $\Str$ to $\Omega_\eps$ we will always assume that
\begin{enumerate}[label=(\Alph*),resume]
	\item the map $\Phi_\eps$ is injective.
\end{enumerate}

We are interested in the spectrum of the Dirac operator with infinite mass boundary conditions posed in the domain $\Omega_\eps$. Let $\cD_\Gamma(\eps)$ denote this operator, it writes
\begin{multline}\label{eq:operator}
	\cD_\Gamma(\eps) := - i \sigma_1 \partial_1 - i\sigma_2 \partial_2 + m \sigma_3\\\dom{\cD_\Gamma(\eps)} := \{ u  \in H^1(\Omega_\eps,\C^2):-i \sigma_3 \sigma \cdot \nu_\eps = u \text{ on }\partial\Omega_\eps \}\,,
\end{multline}
where $\nu_\eps$ is the outward pointing normal vector field on $\partial\Omega_\eps$ and $m\geq 0$ is a fixed parameter. Here $\sigma= (\sigma_1,\sigma_2)$ and $\sigma_1, \sigma_2, \sigma_3$ are the Pauli matrices
\begin{equation}\label{eq:Pauli}
	\sigma_1 := \begin{pmatrix} 0 & 1 \\ 1 & 0\end{pmatrix},\quad \sigma_2 := \begin{pmatrix} 0 & -i \\ i & 0\end{pmatrix},\quad \sigma_3 := \begin{pmatrix}1 &0\\0 &- 1\end{pmatrix}.
\end{equation}
Thanks to \cite[Thm. 2]{BBKOB22}, we know that for $\eps$ small enough, $\cD_\Gamma(\eps)$ is self-adjoint, that its spectrum is symmetric with respect to $0$ and that is essential spectrum is given by
\[
	\operatorname{Sp}_{ess}(\cD_\Gamma(\eps)) = \Big(-\infty, -\frac{E_1(m\eps)}{\eps}\Big]\cup \Big[\frac{E_1(m\eps)}{\eps},+\infty\Big)
\]
where for $\rho \geq0$, $E_1(\rho) := \sqrt{\rho^2 + k_1(\rho)^2}$ and where $k_1(\rho)$ is defined as the unique root of 
\[
	\rho \sin(2k) + k \cos(2k) = 0
\]
lying in $[\frac\pi4,\frac\pi2)$.
\begin{rem}
Notice that there is a slight change in notation, compared to \cite[Thm. 2]{BBKOB22}. Indeed, there $k_1(\cdot)$ is denoted by $E_1(\cdot)$ and then the thresholds of the essential spectrum are $\pm\sqrt{\varepsilon^{-2}E_1(m\varepsilon)+m^2}$.
\end{rem}

Our aim is to investigate the possible existence of discrete spectrum of $\cD_\Gamma(\eps)$ in the thin waveguide regime $\eps \to 0$. To do so, we use the min-max principle for the operator $\cD_\Gamma(\eps)^2$, recalled below.
\begin{defin} Let $Q$ be a closed, lower semi-bounded and densely defined quadratic form with domain $\dom{Q}$ in a Hilbert space $\cH$. For $n \in \N$, the $n$-th min-max value of $Q$ is defined as
\begin{equation}
	\mu_n(Q) := \inf_{\substack{F \subset \dom{Q}\\\dim F = n}} \sup_{u \in F \setminus\{0\}} \frac{Q[u]}{\|u\|_{\cH}^2}.
	\label{eqn:min-maxdef}
\end{equation}
If $A$ is the unique self-adjoint operator associated with the sesquilinear form derived from $Q$ {\it via} Kato's first representation theorem (see \cite[Ch. VI, Thm. 2.1]{Kat}) we shall refer to \eqref{eqn:min-maxdef} as the $n$-th min-max level of $A$ and we note $\mu_n(A) = \mu_n(Q)$.
\end{defin}

Now, we can recall the min-max principle.
\begin{prop}
	Let $Q$ be a closed, lower semi-bounded and densely defined quadratic form with domain $\dom{Q}$ in a Hilbert space $\cH$. Let $A$ be the unique self-adjoint operator associated with $Q$. Then, for $n \in \N$, the following alternative holds true:
	\begin{enumerate}
		\item either $\mu_n(A) < \inf \operatorname{Sp}_{ess}(A)$ and  $\mu_n(A)$ is the $n$-th eigenvalue of $A$ (counted with multiplicities),
		\item or $\mu_n(A) = \inf \operatorname{Sp}_{ess}(A)$ and for all $k \geq n$ there holds $\mu_k(A) = \inf \operatorname{Sp}_{ess}(A)$.
	\end{enumerate}
	\label{prop:min-max}
\end{prop}

In order to state the main result of this paper, we introduce the one dimensional Schrödinger operator defined through its quadratic form as
\begin{equation}
	q_{e}[f] = \int_\R (|f'|^2 - \frac{\kappa^2}{\pi^2}|f|^2) ds,\quad \dom{q_{e}} = H^1(\R)
	\label{eqn:fqeff}
\end{equation}
and set $J := \sharp \{\mu_n(q_e) < 0\}$.
\begin{thm} If $J \geq 1$ then there exists $\eps_1 > 0$ such that for all $\eps \in (0,\eps_1)$ there holds
\[
	 \operatorname{Sp}_{dis}(\cD_\Gamma(\eps)) \neq \emptyset.
\]
Moreover, if $\lambda_j(\cD_\Gamma(\eps))$ denotes the $j$-th positive discrete eigenvalue of $\cD_\Gamma(\eps)$ counted with multiplicity then for all $j \in \{1,\dots,J\}$, there holds
\begin{equation}\label{eq:evasymp}
	\lambda_j(\cD_\Gamma(\eps)) = \frac{E_1(m\eps)}{\eps} + \frac{2}\pi\mu_{j}(q_e)\eps + \mathcal{O}(\varepsilon^2).
\end{equation}
\label{thm:main}
\end{thm}

\begin{rem}A situation in which $J\geq 1$ is when  $\kappa \in L^2(\R)\setminus\{0\}$. Indeed, let $\theta > 0$ and consider the map defined for $s \in \R$ by
\[
	\psi_\theta(s) : = \left\{	\begin{array}{lcl}
							\theta^{-1}(s+2\theta) & \text{if} & s\in [-2\theta,-\theta), \\
							1 & \text{if} & s\in [-\theta,\theta],\\
							-\theta^{-1}(s-2\theta) & \text{if} & s \in (\theta,2\theta],\\
							0& \text{otherwise}.
						\end{array}\right.
\]
One remarks that $\psi_\theta \in H^1(\R)$, verifies $\|\psi_\theta\|_{L^\infty(\R)} \leq 1$ and that
\[
	q_e[\psi_\theta] = \frac2\theta - \frac1{\pi^2}\int_\R \kappa^2 |\psi_\theta|^2ds\leq \frac2\theta - \frac{1}{\pi^2}\int^\theta_{-\theta}\kappa^2\,ds\,.
\]
Hence, since $\kappa\in L^2(\R)$, choosing $\theta$ sufficiently large, we get  $q_e[\psi_\theta] < 0$ and the min-max principle (Prop. \ref{prop:min-max}) gives $\mu_1(q_e) < 0$.
\end{rem}

As already remarked, Theorem \ref{thm:main} proves that as long as the curvature creates bound states for the effective operator given by the quadratic form $q_e$, it also creates bound states for the operator $\cD_\Gamma(\eps)$. Note that in Theorem \ref{thm:main} there is no term of order $0$ and the bound states are at a distance of order $\eps$ from the essential spectrum. This differs from the non-relativistic counter part of this problem studied in \cite[Thm. 5.1.]{DE95}, where the splitting of eigenvalues appears at constant order. However, the result is consistent with \cite[Thm. 4]{BBKOB22} where the authors prove that up to a unitary map, the operator \eqref{eq:operator}, suitably renormalized, behave at constant order as a massive (free) Dirac operator on the real line with effective mass $\frac2\pi m$. The spectrum of this operator being purely absolutely continuous, this is consistent with \eqref{eq:evasymp}, where the constant term arising in the expansion of $\frac{E_1(m\eps)}{\eps^2}$ is precisely given by such effective mass $\frac2\pi m$.


\subsection*
{\bf Acknowledgements.}
We acknowledge the support of Istituto Nazionale di Alta Matematica ``F. Severi", through the Intensive Period ``INdAM Quantum Meetings (IQM22)''. W.B. has been partially supported by Gruppo Nazionale per Analisi Matematica, la Probabilit\`a e le loro Applicazioni (GNAMPA) of the Istituto Nazionale di Alta Matematica (INdAM), through the 2022 project ``Modelli matematici con
singolarit\`a per fenomeni di
interazione  ''.

\section{Preliminaries}
The purpose of \S \ref{subsec:pr1} and \S \ref{pr:2} is to gather several results on one dimensional operators which play an important role in the proof of Theorem \ref{thm:main}. This proof is based on the min-max principle applied to the quadratic form of the square of $\cD_\Gamma(\eps)$, exploiting suitable lower and upper bounds for it, given in \S \ref{pr:3}.

\subsection{The effective operator}\label{subsec:pr1}
In what follows, we deal with the following one dimensional operator, defined through its quadratic form by
\begin{equation}\label{eq:effectiveform}
	\tilde{q}_{e}[f] = \int_\R \Big(|f' - i\frac{\kappa}\pi \sigma_3 f|^2 - \frac{\kappa^2}{\pi^2}|f|^2\Big)ds,\quad \dom{\tilde{q}_e} = H^1(\R,\C^2).
\end{equation}
It turns out that the spectral properties of the operator associated with $\tilde{q}_e$ are related to the one of the operator associated with $q_e$ defined in \eqref{eqn:fqeff}. Notice that the former is defined for vector valued functions, while the latter is defined for scalar ones.

\begin{prop} There exists a unitary map $U : L^2(\R,\C^2)\to L^2(\R,\C^2)$ such that for all $f = (f^+,f^-)^\top \in \dom{(q_e\oplus q_e)}$
\[
	(q_e\oplus q_e)[f] = q_e[f^+] + q_e[f^-]=  \tilde{q}_e[Uf].
\]
\label{prop:ueq1D}
\end{prop}
\begin{proof} Let us consider the following gauge transform
\[
	U : L^2(\R^2,\C^2) \to L^2(\R^2,\C^2),\quad (Uf) = e^{i \frac{\rho}{2}\sigma_3}f,
\]
where for all $s\in \R$ we have set $\rho(s) = \int_0^s \kappa(\eta) d\eta$. Remark that there holds
\[
	|(Uf)' - i \frac{\kappa}2 \sigma_3(Uf)|= |i \frac{\kappa}2 \sigma_3e^{i \frac{\rho}{2}\sigma_3}f + e^{i \frac{\rho}{2}\sigma_3}f' - i \frac{\kappa}2 \sigma_3e^{i \frac{\rho}{2}\sigma_3}f| = |e^{i \frac{\rho}{2}\sigma_3}f'| = |f'|
\]
because for all $s\in \R$, $e^{i \frac{\rho(s)}{2}\sigma_3}$ is a unitary matrix. For the same reason, there holds $|(Uf)| = |f|$ and this yields
\[
	(q_e \oplus q_e) [f] = \tilde{q}_e[f].
\]
\end{proof}
\subsection{The transverse Dirac operator}\label{pr:2}
When proving Theorem \ref{thm:main}, we need to use some spectral properties of a one dimensional operator. It is defined for $m \geq 0$ by
\begin{multline*}
\cT(m) := -i\sigma_2 \frac{d}{dt} + m\sigma_3,\\ \dom{\cT(m)} := \{ u = (u_1,u_2)^\top \in H^1((-1,1),\C^2) : u_2(\pm 1) = \mp u_1(\pm 1)\}.
\end{multline*}
The following proposition holds true.

\begin{prop} Let $m \geq 0$. The operator $\cT(m)$ is self-adjoint and has compact resolvent. Moreover the following holds:
\begin{enumerate}
	\item\label{itm:1-1D}for all $u \in \dom{\cT(m)}$ there holds
		\begin{equation}\label{eq:transverseform}
			\|\cT(m) u\|^2_{L^2(-1,1)} = \|u'\|^2_{L^2(-1,1)} + m^2 \|u\|^2_{L^2(-1,1)} + m (|u(1)|^2 + |u(-1)|^2),
		\end{equation}
	\item\label{itm:2-1D}$Sp\big(\cT(m)\big)\cap [-m,m] = \emptyset$,
	\item\label{itm:4-1D} for all $p\geq 1$, define $k_p(m)$ as the only root lying in $[(2p-1)\frac{\pi}{4},p\frac{\pi}2]$ of
	\[
		m\sin(2 k) + k\cos(2 k) = 0,
	\]
	now if one sets $E_p(m) = \sqrt{m^2 + k_p(m)^2}$ there holds ${Sp\big(\cT(m)\big) = \bigcup_{p\geq 1}\{\pm E_p(m)\}}$,
	\item\label{itm:5-1D} there holds
	\[
	k_1(m) = \frac\pi4 + \frac2\pi m - \frac{16}{\pi^3}{m^2} + \mathcal{O}(m^3),
	\]
	\item\label{itm:6-1D} for $p\geq 1$, a normalized eigenfunction associated with $E_p(m)$ is given by
	\[
		\varphi_p^{m,+}(t) := N_{m,p}\Big(k_p \cos(k_p(t+1))\begin{pmatrix}1\\1\end{pmatrix} + \sin(k_p(t+1))\begin{pmatrix}E_p +m\\-(E_p - m)\end{pmatrix}\Big)
	\]
	where $N_{m,p}$ is a normalization constant. We consider $\varphi_p^{m,-} := \sigma_1 \varphi_p^{m,+}$; a normalized eigenfunction associated with $-E_p(m)$ and if one sets $\varphi_p^{\pm} := \varphi_p^{0,\pm}$ there holds
	\[
		\varphi_1^{m,\pm} = \varphi_1^{\pm} + \mathcal{O}(m),
	\]
	where the remainder is understood in the $L^\infty$-norm on $(-1,1)$.
\end{enumerate}
\label{prop:1D}
\end{prop}
\begin{proof} The proof of Points \eqref{itm:1-1D}-\eqref{itm:4-1D} can be found, {\it e.g.}, in \cite[Proposition 10]{BBKOB22}. Point \eqref{itm:5-1D} relies on the fact that
\begin{equation}
	m\sin(2k_1(m)) + k_1(m)\cos(2k_1(m)) = 0.
	\label{eqn:implicit}
\end{equation}
Hence, as $k_1$ is defined near $m=0$ by this smooth implicit equation, $k_1$ is smooth near $m = 0$ and there holds
\[
	k_1(m) = k_1(0) + k_1'(0) m + \frac{1}2k_1''(0)m^2 + \mathcal{O}(m^3),\quad m \to 0.
\]
One can compute thanks to \eqref{eqn:implicit} that
\[
	k_1(0) = \frac\pi4,\quad k_1'(0) = \frac2\pi,\quad k_1''(0) = - \frac{32}{\pi^3},
\]
which yields Point \eqref{itm:5-1D}. To prove Point \eqref{itm:6-1D}, again by \cite[Proposition 10]{BBKOB22}, any eigenfunction associated with $E_p(m)$ is of the form
\[
	u_p^m(t) = \cos(k_p(m)(t+1)) \begin{pmatrix}\alpha\\\frac{k_p(m)}{E_p(m)+m}\beta\end{pmatrix} + \sin(k_p(m)(t+1))\begin{pmatrix}\beta\\ - \frac{k_p(m)}{E_p(m)+m}\alpha\end{pmatrix},
\]
for some constants $\alpha,\beta \in \C$. The boundary condition at $t = - 1$ gives ${\alpha = \frac{k_p(m)}{E_p(m)+m}\beta}$ so that, choosing $\beta = (E_p(m) +m)$:
\[
	u_p^m(t) = k_p(m)\cos(k_p(m)(t+1)) \begin{pmatrix}1\\1\end{pmatrix} + \sin(k_p(m)(t+1))\begin{pmatrix}E_p(m)+m\\ -(E_p(m) - m)\end{pmatrix}.
\]
Hence, we take $N_{m,p} := \|u_p^m\|_{L^2((-1,1),\C^2)}^{-1}$ and remark that $\varphi_p^{m,+} := N_{m,p} u_p^m$. Note that
\begin{multline*}
	\|u_1^m\|_{L^2((-1,1),\C^2)}^2 = 2 k_1(m)^2 \big(1 + \frac{\sin(4 k_1(m))}{4k_1(m)}\big) + 2 (E_1(m)^2 +m^2) \big(1 - \frac{\sin(4 k_1(m))}{4k_1(m)}\big) \\+  2 m k_1(m) \big(1 - \frac{\cos(4k_1(m))}{2k_1(m)}\big)
\end{multline*}
which gives using Point \eqref{itm:5-1D}
\begin{equation}
	N_{1,m} = \frac{1}{2k_1(m)} + \mathcal{O}(m),\quad m \to 0.
	\label{eqn:danorm}
\end{equation}
Now, remark that for all $t \in (-1,1)$ there holds
\[
	|\varphi_1^{m,+}(t) - \varphi_1^+(t)| \leq 2\Big|k_1(m) N_{m,p} - \frac12\Big| + 2\Big|E_1(m) N_{m,p} - \frac12\Big| + 2m.
\]
which gives Point \eqref{itm:6-1D} for $\varphi_1^{m,+}$ thanks to \eqref{eqn:danorm}. For $\varphi_1^{m,-}$ one only has to note that for all $t\in(-1,1)$ there holds
$|\varphi_1^{m,-}(t) - \varphi_1^-(t)| = |\sigma_1(\varphi_1^{m,-}(t) - \varphi_1^-(t))| = |\varphi_1^{m,+}(t) - \varphi_1^+(t)|$.
\end{proof}
\begin{rem} The explicit expression of the functions $\varphi_1^\pm$ is of crucial importance in what follows. They are defined, for all $t \in (-1,1)$, as\begin{equation}
	\label{eqn:defphipm}
	\varphi_1^\pm(t) = \frac12 \cos(\frac\pi4(t+1))\begin{pmatrix} 1 \\ 1 \end{pmatrix} \pm \frac12 \sin(\frac\pi4(t+1)) \begin{pmatrix}1 \\ - 1\end{pmatrix}.
\end{equation}
\end{rem}
\subsection{The quadratic form of the square}\label{pr:3}
By \cite[Prop. 3]{BBKOB22} we know that the operator $\cD_\Gamma(\varepsilon)$ is unitarily equivalent to
\begin{multline*}
\cE_\Gamma(\varepsilon) := \frac1{1-\varepsilon t \kappa}(-i\sigma) \partial_s + \frac1\varepsilon(-i\sigma_2)\partial_t + \frac{\varepsilon t \kappa'}{2(1-\varepsilon t \kappa)^2}(-i\sigma_1) + m\sigma_3,\\
\dom{\cE_\Gamma(\varepsilon) := \{u = (u_1,u_2)^\top \in H^1(\Str,\C^2) : u_2(\cdot,\pm 1) =\mp u_1(\cdot,\pm 1)\}}
\end{multline*}
and that the quadratic form of its square is given, for every
$u \in \dom{\cE_\Gamma(\eps)}$, by
\begin{equation}\label{eqn:fqunit}\begin{split}
\|\cE_\Gamma(\eps) u\|_{L^2(\Str,\C^2)}^2= \ &\int_{\Str}\frac1{(1-\varepsilon t \kappa)^2}\vert\partial_s u - i\frac\kappa2\sigma_3 u\vert^2 ds dt + \frac1{\varepsilon^2}\int_\Str \vert\partial_t u\vert^2 dsdt\\
&+\frac{m}{\eps}\int_{\R}\left(\vert u(s,1)\vert^2+ \vert u(s,-1)\vert^2 \right)ds+m^2\|u\|_{L^2(\Str,\C^2)}^2 \\
&- \int_{\Str}\frac{\kappa^2}{4(1-\varepsilon t \kappa)^2}\vert u\vert^2 ds dt-\frac{5}{4}\int_{\Str}\frac{(\eps t\kappa')^2}{(1-\eps t\kappa)^4}\vert u\vert^2dsdt\\&-\frac{1}{2}\int_{\Str}\frac{\eps t\kappa''}{(1-\eps t\kappa)^3} \vert u\vert^2dsdt\,.
\end{split}
\end{equation}
The main result of this section reads as follows. 
\begin{prop} There exists $\varepsilon' > 0$ and $c > 0$ such that for all $\eps \in (0,\eps')$ and all $u \in \dom{\cE_\Gamma(\varepsilon,m)}$ there holds
\begin{equation}\label{eq:bounds}
	a_-[u] \leq  \|\cE_\Gamma(\eps) u\|_{L^2(\Str,\C^2)}^2 \leq a_+[u],
\end{equation}
where we have introduced the quadratic forms $a_\pm$ defined by
\begin{multline*}
	a_\pm[u] := (1\pm c\varepsilon)\int_\Str\Big(|\partial_s u - i\frac{\kappa}{2}\sigma_3 u|^2 - \frac{\kappa^2}{4}|u|^2\Big) ds dt + \frac1{\varepsilon^2}\int_\R \big(\|\cT(m\varepsilon)u\|_{L^2((-1,1),\C^2)}^2 \big)ds \pm c \varepsilon \|u\|^2,\\\dom{a_\pm} := \dom{\cE_\Gamma(\varepsilon)}.
\end{multline*}
\label{prop:encfq}
\end{prop}
The proof of Proposition \ref{prop:encfq} is straightforward taking into account \eqref{eqn:fqunit} and the fact that $\kappa,\kappa', \kappa'' \in L^\infty(\R)$.  To this aim, observe that $(1-\varepsilon t\kappa(s))^{-1} = 1+\mathcal O(\varepsilon)$, uniformly in $(s,t)\in \R\times[-1,1]$, and recall \eqref{eq:transverseform}. The bounds \eqref{eq:bounds} will be used to get the desired spectral asymptotics, by comparison.
\section{Proof of the main result}
In \S \ref{pr:ub} we give an upper bound on the $j$-th min-max level of $\cD_\Gamma(\eps)^2$, while a lower bound is obtained in \S \ref{pr:lb}. Combining these results, Theorem \ref{thm:main} is proved in \S \ref{pr:pr}.

\subsection{An upper bound}\label{pr:ub}
The goal of this paragraph is to prove the following Proposition.
\begin{prop} Let $j\in\N$. There exists $\varepsilon_1 > 0$ and $c > 0$ such that for all $\varepsilon \in (0,\varepsilon_1)$ there holds
\begin{equation}\label{eq:upbound}
	\mu_j(\cD_\Gamma(\varepsilon)^2) \leq \frac{E_1(m\varepsilon)^2}{\eps^2} + \mu_j(q_e\oplus q_e) + c \varepsilon.
\end{equation}
\label{prop:ub}
\end{prop}

\begin{proof} Let $f = (f^+,f^-) \in H^1(\R, \C^2)$ and set $u = f^+\varphi_1^{m\varepsilon,+} + f^- \varphi_1^{m\varepsilon,-}$. By construction $u \in \dom{\cE_\Gamma(\varepsilon)}$ and for $\eps$ small enough there holds
\[
	a_+[u] = (1+c\varepsilon) \int_\Str\Big( |\partial_s u - i\frac{\kappa}2\sigma_3 u|^2\Big)dsdt - (1+c\varepsilon)\int_\R \frac{\kappa^2}4 |f|^2 ds + \frac{E_1(m\varepsilon)}\varepsilon^2 \|f\|^2_{L^2(\R)} + c\varepsilon \|f\|^2_{L^2(\R)}.
\]
Now, one remarks that
\[
	\int_\Str\Big( |\partial_s u - i\frac{\kappa}2\sigma_3 u|^2\Big)ds dt = \int_\R |f'|^2ds + \int_\R \frac{\kappa^2}4 |f|^2 ds + \int_\R \kappa \Re\Big(\int_{-1}^1 \langle\partial_s u, - i\sigma_3 u\rangle\,dt\Big)ds
\]
and there holds
\begin{multline*}
	\langle\partial_s u, - i\sigma_3 u\rangle = (f^+)'\overline{f^+}\langle\varphi_1^{m\varepsilon,+},-i\sigma_3 \varphi_1^{m\varepsilon,+}\rangle + (f^+)'\overline{f^-}\langle\varphi_1^{m\varepsilon,+},-i\sigma_3 \varphi_1^{m\varepsilon,-}\rangle \\+ (f^-)'\overline{f^+}\langle\varphi_1^{m\varepsilon,-},-i\sigma_3 \varphi_1^{m\varepsilon,+}\rangle +(f^-)'\overline{f^-}\langle\varphi_1^{m\varepsilon,-},-i\sigma_3 \varphi_1^{m\varepsilon,-}\rangle.
\end{multline*}
Now by Point \eqref{itm:6-1D} in Proposition \ref{prop:1D}, there holds
\[
	\int_{-1}^1\langle\varphi_1^{m\varepsilon,+},-i\sigma_3 \varphi_1^{m\varepsilon,+}\rangle dt = i\int_{-1}^1\langle\varphi_1^{+},\sigma_3 \varphi_1^{+}\rangle dt + \mathcal{O}(\varepsilon) = i\frac2\pi + \mathcal{O}(\varepsilon),
\]
where we have used the explicit expression of $\varphi_1^+$ given in \eqref{eqn:defphipm}.
Similarly, one gets
\[
\int_{-1}^1\langle\varphi_1^{m\varepsilon,-},-i\sigma_3 \varphi_1^{m\varepsilon,-}\rangle dt = -i\frac2\pi + \mathcal{O}(\varepsilon)
\]
as well as
\[
	\int_{-1}^1\langle\varphi_1^{m\varepsilon,+},-i\sigma_3 \varphi_1^{m\varepsilon,-}\rangle \,dt = \mathcal{O}(\varepsilon),\quad \int_{-1}^1\langle\varphi_1^{m\varepsilon,-},-i\sigma_3 \varphi_1^{m\varepsilon,+}\rangle\,dt = \mathcal{O}(\varepsilon).
\]
Hence, we get
\[
	\int_{-1}^1\langle\partial_s u, - i\sigma_3 u\rangle \,dt = i\frac2\pi\big((f^+)'\overline{f^+} - (f^{-})'\overline{f^-}\big) + \langle f',f + \sigma_1 f\rangle\mathcal{O}(\varepsilon) = \langle\partial_s f, -i\frac{2}\pi \sigma_3 f\rangle + \langle f',f + \sigma_1 f\rangle\mathcal{O}(\varepsilon).
\]
Thus, there exists $\varepsilon_1 > 0$ small enough and $k > 0$ such that for all $\eps \in (0,\eps_1)$ there holds
\[
	a_+[u] \leq (1+c\varepsilon) \Big(\int_\R |f'|^2 ds + \int_\R 2\Re(\langle f',-i\frac{\kappa}\pi \sigma_3 f\rangle) ds\Big) +\frac{E_1(m\varepsilon)^2}{\varepsilon^2} \|f\|^2_{L^2(\R)} + (1+c\varepsilon)k\varepsilon(\|f\|^2_{L^2(\R)} + \|f'\|^2_{L^2(\R)}).
\]
Now, remark that
\[
	\int_\R |f'|^2 ds + \int_\R 2\Re(\langle f',-i\frac{\kappa}\pi \sigma_3 f\rangle) ds = \int_\R \Big(|f' - i \frac{\kappa}{\pi}\sigma_3 f|^2 - \frac{\kappa^2}{\pi^2}|f|^2 \Big)ds.
\]
Thus, for a $c_1 >0$ there holds
\begin{multline*}
	a_+[u] \leq (1+c_1\varepsilon) \int_\R \Big(|f' - i \frac{\kappa}{\pi}\sigma_3 f|^2 - \frac{\kappa^2}{\pi^2}|f|^2 \Big)ds + \frac{E_1(m\varepsilon)^2}{\varepsilon^2} \|f\|^2_{L^2(\R)} + c_1 \varepsilon \|f\|^2_{L^2(\R)} \\= (1+c_1\varepsilon) \tilde{q}_e[f] +\frac{E_1(m\varepsilon)^2}{\varepsilon^2} \|f\|^2_{L^2(\R)} + c_1 \varepsilon \|f\|^2_{L^2(\R)}.
\end{multline*}
Now, the min-max principle of Proposition \ref{prop:min-max}, Proposition \ref{prop:encfq} and Proposition \ref{prop:ueq1D} give for all $j \in \N$:
\[
	\mu_j(\cD_\Gamma(\eps)^2) \leq (1+c_1\varepsilon) \mu_j(q_e\oplus q_e) + \frac{E_1(m\varepsilon)^2}{\varepsilon^2} + c_1\varepsilon 
\]
so that \eqref{eq:upbound} follows.
\end{proof}
\subsection{A lower bound}\label{pr:lb}
The aim of this paragraph is to prove the following lower bound.

\begin{prop}
	Let $j \in \N$. There exists $\varepsilon_1 > 0$ and $c > 0$ such that for all $\varepsilon \in (0,\varepsilon_1)$ there holds
	\begin{equation}\label{eq:lowbound}
		  \mu_j(\cD_\Gamma(\varepsilon)^2)\geq \frac{E_1(m\varepsilon)^2}{\eps^2} +\mu_j(q_e\oplus q_e) - c\varepsilon \,.
	\end{equation}
	\label{prop:lb}
\end{prop}
To prove Proposition \ref{prop:lb}, we need to introduce the projector in $L^2(\Str,\C^2)$ defined for all $\delta > 0$ and $u \in L^2(\Str,\C^2)$ by
\[
	\Pi^{\delta} u := \langle u,\varphi_1^{\delta,+}\rangle_{L^2(-1,1)}\varphi_1^{\delta,+} + \langle u,\varphi_1^{\delta,-}\rangle_{L^2(-1,1)}\varphi_1^{\delta,-}
\]
and set $\Pi := \Pi^0$. Thanks to Point \ref{itm:6-1D} in Proposition \ref{prop:1D}, there holds
\[
	\Pi^\delta = \Pi + \mathcal{O}(\delta), \qquad\mbox{ as $\delta \to 0$,}
\]
where the remainder is estimated in the operator norm. We also set $(\Pi^{\delta})^\perp := Id - \Pi^\delta$ and $\Pi^\perp = Id - \Pi$.
\begin{proof}
Let $u \in \dom{\cE_\Gamma(\varepsilon)}$ and remark that there holds
\begin{multline}
	a_-[u] - \frac{E_1(m\varepsilon)^2}{\eps^2} \|u\|^2\geq (1-c\varepsilon)\int_\Str\Big(|(\partial_s - i \frac{\kappa}2 \sigma_3)(\Pi^{m\eps} + (\Pi^{m\eps})^\perp)u|^2 \\- \frac{\kappa^2}4 (|\Pi^{m\eps} u|^2 + |(\Pi^{m\eps})^\perp u|^2)\Big)ds dt  \\+ \frac{E_2(m\varepsilon)^2-E_1(m\varepsilon)^2}{\eps^2}\|(\Pi^{m\eps})^\perp u\|^2 \\- c\varepsilon (\|\Pi^{m\varepsilon} u\|^2 + \|(\Pi^{m\varepsilon})^\perp u\|^2).
	\label{eqn:1lb}
\end{multline}
We focus on the first term on the right-hand side of the last equation which gives
\begin{multline*}
\int_\Str \big(|\big(\partial_s - i \frac{\kappa}2 \sigma_3\big)\big(\Pi^{m\varepsilon} + (\Pi^{m\varepsilon})^\perp\big)u|^2\big)ds dt = \int_\Str |\big(\partial_s - i \frac{\kappa}2 \sigma_3\big)\Pi^{m\varepsilon}u|^2 ds dt \\+ \int_\Str |\big(\partial_s - i \frac{\kappa}2 \sigma_3\big)(\Pi^{m\varepsilon})^\perp u|^2 ds dt \\+  2\Re\Big(\int_\Str \langle\big(\partial_s  - i \frac{\kappa}2 \sigma_3\big)\Pi^{m\eps} u, \big(\partial_s  - i \frac{\kappa}2 \sigma_3\big)(\Pi^{m\eps})^\perp u\rangle ds dt\Big)\\\geq \int_\Str |\big(\partial_s  - i \frac{\kappa}2 \sigma_3\big)\Pi^{m\varepsilon}u|^2 ds dt\\+ 2\Re\Big(\int_\Str \langle\big(\partial_s  - i \frac{\kappa}2 \sigma_3\big)\Pi^{m\eps} u, \big(\partial_s  - i \frac{\kappa}2 \sigma_3\big)(\Pi^{m\eps})^\perp u\rangle ds dt\Big).
\end{multline*}
Let us deal with the last term. Remark that there holds
\begin{multline*}
	\langle\big(\partial_s  - i \frac{\kappa}2 \sigma_3\big)\Pi^{m\eps} u, \big(\partial_s  - i \frac{\kappa}2 \sigma_3\big)(\Pi^{m\eps})^\perp u\rangle\\= \langle\big([\big(\partial_s  - i \frac{\kappa}2 \sigma_3\big),\Pi^{m\eps}] + \Pi^{m\eps}\big(\partial_s  - i \frac{\kappa}2 \sigma_3\big)\big)\Pi^{m\eps}u,\big([\big(\partial_s  - i \frac{\kappa}2 \sigma_3\big),(\Pi^{m\eps})^\perp] + (\Pi^{m\eps})^\perp\big(\partial_s  - i \frac{\kappa}2 \sigma_3\big)\big)(\Pi^{m\eps})^\perp u\rangle,
\end{multline*}
where the scalar product is taken in $L^2(\Str,\C^2)$. Taking into account that $\Pi^{m\eps}$ commutes with $\partial_s$, we obtain
\[
	[\big(\partial_s  - i \frac{\kappa}2 \sigma_3\big),\Pi^{m\eps}] = - [\big(\partial_s  - i \frac{\kappa}2 \sigma_3\big),(\Pi^{m\eps})^\perp] = - i \frac{\kappa}{2}[\sigma_3,\Pi^{m\eps}],
\]
which gives
\begin{multline*}
\langle\big(\partial_s  - i \frac{\kappa}2 \sigma_3\big)\Pi^{m\eps} u, \big(\partial_s  - i \frac{\kappa}2 \sigma_3\big)(\Pi^{m\eps})^\perp u\rangle\\= \langle\big(- i \frac{\kappa}{2}[\sigma_3,\Pi^{m\eps}] + \Pi^{m\eps}\big(\partial_s  - i \frac{\kappa}2 \sigma_3\big)\big)\Pi^{m\eps}u,\big( i \frac{\kappa}{2}[\sigma_3,\Pi^{m\eps}] + (\Pi^{m\eps})^\perp\big(\partial_s  - i \frac{\kappa}2 \sigma_3\big)\big)(\Pi^{m\eps})^\perp u\rangle\\
= \langle - i \frac{\kappa}{2}[\sigma_3,\Pi^{m\eps}] \Pi^{m\eps}u, i \frac{\kappa}{2}[\sigma_3,\Pi^{m\eps}](\Pi^{m\eps})^\perp u\rangle + \langle - i \frac{\kappa}{2}[\sigma_3,\Pi^{m\eps}] \Pi^{m\eps}u, (\Pi^{m\eps})^\perp\big(\partial_s  - i \frac{\kappa}2 \sigma_3\big)(\Pi^{m\eps})^\perp u\rangle\\ + \langle\Pi^{m\eps}\big(\partial_s  - i \frac{\kappa}2 \sigma_3\big)\Pi^{m\eps}u, i\frac\kappa2[\sigma_3,\Pi^{m\eps}](\Pi^{m\eps})^\perp u\rangle\\ := J_1 + J_2 + J_3.
\end{multline*}
One notices that there exists $c_1 > 0$ such that
\begin{equation}
	|J_1| \leq c_1 \|\Pi^{m\eps}u\| \|(\Pi^{m\eps})^\perp u\| \leq \frac{c_1}2\eps \|\Pi^{m\eps}u\|^2 + \frac{c_1}{2\eps}\|(\Pi^{m\eps})^\perp u\|^2.
	\label{eqn:J1}
\end{equation}
Similarly, there exists $c_2 > 0$ such that
\begin{equation}
	|J_3| \leq c_2 \|(\partial_s - i\frac{\kappa}{2})\Pi^{m\eps} u\| \|(\Pi^{m\eps})^\perp u\|\leq \frac{c_2}{2}\eps\|(\partial_s - i\frac{\kappa}{2})\Pi^{m\eps} u\|^2 + \frac{c_2}{2\eps}\|(\Pi^{m\eps})^\perp u\|^2.
	\label{eqn:J3}
\end{equation}

Concerning the term $J_2$, there holds
\[
	J_2 = \langle - i\frac\kappa2 (\Pi^{m\eps})^\perp \sigma_3 \Pi^{m\eps} u, \partial_s (\Pi^{m\eps})^\perp u\rangle + \langle - i \frac\kappa2 (\Pi^{m\eps})^\perp \sigma_3 \Pi^{m\eps} u, -i \frac\kappa2\sigma_3 (\Pi^{m\eps})^\perp u\rangle
\]
and an integration by parts in the $s$-variable gives
\begin{multline*}
	J_2 = -\langle - i \frac\kappa2 (\Pi^{m\eps})^\perp \sigma_3 \Pi^{m\eps} \partial_s u, (\Pi^{m\eps})^\perp u\rangle - \langle - i \frac{\kappa'}2(\Pi^{m\eps})^\perp\sigma_3\Pi^{m\eps}u,(\Pi^{m\eps})^\perp u\rangle\\+ \langle - i \frac\kappa2 (\Pi^{m\eps})^\perp \sigma_3 \Pi^{m\eps} u, -i \frac\kappa2\sigma_3 (\Pi^{m\eps})^\perp u\rangle.
\end{multline*}
Hence, there exists $c_3 > 0$ such that
\begin{multline*}
	|J_2| \leq c_3 \big( \|\partial_s \Pi^{m\eps}u\| \|(\Pi^{m\eps})^\perp u\| + \|\Pi^{m\eps}u\| \|(\Pi^{m\eps})^\perp u\|\big) \\\leq c_3 \big(\frac\eps2 (\|\partial_s \Pi^{m\eps} u\|^2 + \|\Pi^{m\eps}u\|^2) + \frac1{2\eps}\|(\Pi^{m\eps})^\perp u\|^2\big).
\end{multline*}
Noting that 
\[
	\|\partial_s \Pi^{m\eps}u \| \leq \|\big(\partial_s -i\frac\kappa2 \sigma_3\big) u\| + \|\frac\kappa2 \Pi^{m\eps}u\|
\]
there exists $c_4 > 0$ such that
\begin{equation}\label{eq:J_2}
	|J_2| \leq c_4 \big(\frac\eps2 (\|\big(\partial_s -i\frac\kappa2 \sigma_3\big)\Pi^{m\eps}  u\|^2 + \|\Pi^{m\eps}u\|^2) + \frac1{2\eps}\|(\Pi^{m\eps})^\perp u\|^2\big).
\end{equation}
In the estimates \eqref{eqn:J1}, \eqref{eqn:J3} and \eqref{eq:J_2} we have used that $\kappa \in L^\infty(\R)$,  and the following bounds on the operator norms
\[
\|\Pi^{m\eps}\| \leq 1\,,\qquad \|(\Pi^{m\eps})^\perp\|\leq 1\]
as $\Pi^{m\eps}$ and $(\Pi^{m\eps})^\perp$ are projectors, as well as the elementary identity  $ 
2ab \leq \eps a^2 + \frac1{\eps}b^2$, with $a,b,\eps > 0$.

Combining the above observations, coming back to \eqref{eqn:1lb}, we get, for some $c_5 \geq 0$ :
\begin{multline*}
	a_-[u]- \frac{E_1(m\eps)^2}{\eps^2} \geq (1-c_5\eps) \int_\Str \Big(|(\partial_s - i \frac\kappa2\sigma_3) \Pi^{m\eps}u|^2 - \frac{\kappa^2}4|\Pi^{m\eps} u|^2\Big)dtds\\ + \Big(\frac{E_2(m\eps)^2 - E_1(m\eps)^2}{\eps^2} - \frac{c_5}\eps - c_5\Big) \|(\Pi^{m\eps})^\perp u\|^2 - c_5 \|\Pi^{m\eps} u\|^2\,. 
\end{multline*}
Notice that the first term on the right-hand-side of \eqref{eq:J_2} has been absorbed in the integral above, taking a new constant $c_5$, and the terms involving $\Vert \Pi^{m\varepsilon}u\Vert^2$ and $\Vert (\Pi^{m\varepsilon})^\perp u\Vert^2$ contribute to the last two terms in the above formula.

Now, set $f^\pm := \langle u, \varphi_1^{m\eps,\pm}\rangle_{L^2(-1,1)}$ and remark that the computation of the term
\[
	\int_\Str \Big(|(\partial_s - i \frac\kappa2\sigma_3) \Pi^{m\eps}u|^2 - \frac{\kappa^2}4|\Pi^{m\eps} u|^2 ds dt.
\]
is similar to the one performed in the proof of the upper bound (see the proof of Proposition \ref{prop:ub}) and it yields
\[
	\int_\Str \Big(|(\partial_s - i \frac\kappa2\sigma_3) \Pi^{m\eps}u|^2 - \frac{\kappa^2}4|\Pi^{m\eps} u|^2 ds dt \geq (q_e\oplus q_e)[f] - c_6 \varepsilon \|f\|^2_{L^2(\R)},
\]
for some constant $c_6$. All in all, we have obtained that there exists $k > 0$ such that provided $\eps$ is small enough there holds
\begin{multline*}
	a_-[u] - \frac{E_1(m\eps)^2}{\eps^2} \geq (1-k\eps)(q_e\oplus q_e)[f] - k \eps \|f\|^2_{L^2(\R)}\\+ \Big(\frac{E_2(m\eps)^2 - E_1(m\eps)^2}{\eps^2} - \frac{k}\eps - k\Big) \|(\Pi^{m\eps})^\perp u\|^2  \\\geq (1-k\eps)(q_e\oplus q_e)[f] - k \eps \|f\|^2_{L^2(\R)}+ \Big(\frac{5\pi^2}{16\eps^2} - \frac{k}\eps - k\Big) \|(\Pi^{m\eps})^\perp u\|^2,
\end{multline*}
where for the last inequality we have used Point \eqref{itm:4-1D} Proposition \ref{prop:1D}. As the quadratic form on the right-hand side is the quadratic form of the direct sum of two operators, if one fixes $j \in \N$ the min-max principle of Proposition \ref{prop:min-max} yields
\[
	\mu_j(\cD_\Gamma(\eps)^2) -\frac{E_1(m\eps)^2}{\eps^2}\geq j-\text{th element of the set} \Big( \{(1-k\eps)\mu_j(q_e\oplus q_e) - k\eps\}\cup \{\frac{5\pi^2}{16\eps^2} - \frac{k}\eps - k\}\Big).
\]
Hence, for $\eps$ small enough (depending on $j$), this reads
\[
	\mu_j(\cD_\Gamma(\eps)^2) -\frac{E_1(m\eps)^2}{\eps^2}\geq (1-k\eps)\mu_j(q_e\oplus q_e) - k\eps.
\]
which is precisely Proposition \ref{prop:lb}.
\end{proof}
\subsection{Proof of Theorem \ref{thm:main}}\label{pr:pr}
Let $J \geq 1$ and remark that due to the symmetry of the spectrum of $\cD_\Gamma(\eps)$ with respect to zero, for all $j \in \{1,\dots,J\}$ there holds
\begin{equation}
	\lambda_j(\cD_\Gamma(\eps)) = \sqrt{\mu_{2j}(\cD_\Gamma(\eps)^2)}.
	\label{eqn:relcarre}
\end{equation}
Combining Propositions \ref{prop:ub} and \ref{prop:lb}, we have for all $j \in \N$ that when $\eps \to 0$
\begin{multline*}
	\mu_j(\cD_\Gamma(\eps)^2) = \frac{E_1(m\eps)^2}{\eps^2} + \mu_j(q_e\oplus q_e) + \mathcal{O}(\eps) \\= \frac{E_1(m\eps)^2}{\eps^2}\Big(1 + \frac{\eps^2}{E_1(m\eps)^2}\mu_{j}(q_e\oplus q_e) + \mathcal{O}(\varepsilon^3)\Big),
\end{multline*}
where we have used that $E_1(m\eps) = \mathcal{O}(1)$ when $\eps \to 0$. Hence, there holds
\[
	\sqrt{\mu_j(\cD_\Gamma(\eps)^2)} = \frac{E_1(m\eps)}\eps + \frac{1}{2 E_1(m\eps)}\mu_j(q_e\oplus q_e) \eps + \mathcal{O}(\eps^2)
\]
and by Point \ref{itm:5-1D} in Proposition \ref{prop:1D}, there holds
\[
	\sqrt{\mu_j(\cD_\Gamma(\eps)^2)} = \frac{E_1(m\eps)}\eps + \frac2\pi \mu_j(q_e\oplus q_e) \eps + \mathcal{O}(\eps^2).
\]
Thus, for $j\in \{1,\dots,J\}$, \eqref{eqn:relcarre} yields
\[	
	\lambda_j(\cD_\Gamma(\eps)) = \frac{E_1(m\eps)}\eps + \frac2\pi \mu_{2j}(q_e\oplus q_e) \eps + \mathcal{O}(\eps^2) = \frac{E_1(m\eps)}\eps + \frac2\pi \mu_{j}(q_e) \eps + \mathcal{O}(\eps^2),
\]
concluding the proof.
\providecommand{\bysame}{\leavevmode\hbox to3em{\hrulefill}\thinspace}
\providecommand{\MR}{\relax\ifhmode\unskip\space\fi MR }
\providecommand{\MRhref}[2]{%
  \href{http://www.ams.org/mathscinet-getitem?mr=#1}{#2}
}
\providecommand{\href}[2]{#2}

\end{document}